\documentclass[11pt]{amsart} 

\usepackage{amsmath, amsthm, amscd, amsfonts, amssymb, enumerate, verbatim, newlfont, calc, graphicx, color}
\usepackage[bookmarksnumbered, colorlinks, plainpages]{hyperref}
\usepackage{tikz}
 \usepackage{doi}

\usetikzlibrary{shapes.geometric, positioning, arrows.meta}

\usepackage{cancel}
\newtheorem{theorem}{Theorem}[section] 
 
\newtheorem{fact}[theorem]{Fact}

\newtheorem{lemma}[theorem]{Lemma}
\newtheorem{proposition}[theorem]{Proposition}

\newtheorem{corollary}[theorem]{Corollary}
\newtheorem{definition}[theorem]{Definition}

\newtheorem{remark}[theorem]{Remark}

\newtheorem{example}[theorem]{Example}

\usepackage{mathtools}

\usepackage{pstricks}
\usepackage{epsfig}
\usepackage{pst-grad}
\usepackage{pst-plot}
\usepackage{subfig}

\begin{document}

\author[M. Nasernejad  and   J. Toledo]{Mehrdad  Nasernejad$^{1,2}$  and   Jonathan Toledo$^{*,3}$}
\title[Asymptotic Properties of Filtrations of Ideals]{Asymptotic Properties of Filtrations of Ideals}


\thanks{$^*$Corresponding author}

\thanks{E-mail addresses:  m$\_$nasernejad@yahoo.com  and  jonathan.toledo@infotec.mx}  
\maketitle

\begin{center}
{\it
$^{1}$Univ. Artois, UR 2462, Laboratoire de Math\'{e}matique de  Lens (LML), \\  F-62300 Lens, France \\ 
$^{2}$Universit\'e  Caen Normandie, ENSICAEN, CNRS, Normandie Univ, GREYC UMR  6072, F-14000 Caen,  France\\
$^{3}$INFOTEC Centro de investigaci\'{o}n e innovaci\'{o}n en informaci\'{o}n \\
 y comunicaci\'{o}n, Ciudad de M\'{e}xico,14050, M\'{e}xico
}
\end{center}


\begin{abstract}
 We introduce a unified framework for studying persistence phenomena in
commutative algebra via filtrations of ideals. For a filtration
$\mathcal{F}=\{I_i\}_{i\in\mathbb{N}}$, we define $\mathcal{F}$-persistence and
$\mathcal{F}$-strong persistence, extending the classical notions for ordinary
and symbolic powers of ideals. We show that if $\mathcal{F}$ is strongly
persistent, then $\mathcal{F}_{\mathrm{sym}}$ is strongly persistent, 
 where    $\mathcal{F}_{\mathrm{sym}}$   denotes the symbolic filtration  associated with  the 
filtration $\mathcal{F}$. 
 In addition, we prove that if $\mathcal{F}$ is strongly persistent, then
$\mathcal{F}$ is persistent.
\end{abstract}

\maketitle

\section{Introduction}

The study of persistence phenomena for powers and symbolic powers of ideals has played a central role in commutative algebra over the last decades. 
Let $I$ be an ideal of a commutative Noetherian ring $R$. A prime ideal
$\mathfrak{p}\subset R$ is called an \emph{associated prime} of $I$ if there
exists an element $h\in R$ such that $\mathfrak{p}=(I:_R h).$ 
The collection of all prime ideals associated with $I$, denoted by
$\mathrm{Ass}_R(R/I)$, is referred to as the \emph{set of associated primes} of
$I$. Our focus is on the behavior of the sets $\mathrm{Ass}_R(R/I^t)$ as the
exponent $t$ varies. 
A fundamental result due to Brodmann \cite{BR} establishes that the sequence
 $\{\mathrm{Ass}_R(R/I^t)\}_{t\geq 1}$ 
eventually stabilizes. More precisely, there exists a positive integer $t_0$
such that $\mathrm{Ass}_R(R/I^t)=\mathrm{Ass}_R(R/I^{t_0})$ for all $t \geq t_0.$ 
The smallest such integer $t_0$ is known as the \emph{index of stability} of $I$,
and the common value $\mathrm{Ass}_R(R/I^{t_0})$ is called the
\emph{stable set of associated primes} of $I$, denoted by
$\mathrm{Ass}^{\infty}(I)$.

Brodmann’s theorem has motivated a variety of further investigations. One
natural question that arises is whether the inclusion
\[
\mathrm{Ass}_R(R/I^t)\subseteq \mathrm{Ass}_R(R/I^{t+1})
\]
holds for all $t\geq 1$. McAdam \cite{Mc} provided a counterexample showing that
this inclusion does not hold in general. An ideal $I\subset R$ is said to
satisfy the \emph{persistence property} if the above inclusion is valid for
every $t\geq 1$. Moreover, an ideal $I$ is said to have the \emph{strong persistence property}
if
\[
(I^{t+1}:_R I)=I^t \quad \text{for all } t\geq 1.
\]
Nevertheless, Ratliff \cite{RA} proved  that  this equality holds for all
sufficiently large values of $t$. It is important to note that the strong
persistence property implies the persistence property (for instance, see the
proof of \cite[Proposition~2.9]{N1}), whereas the converse implication fails in
general.

In particular, an  ideal $I$ in a commutative Noetherian ring $R$ has the  {\it symbolic strong persistence property} if 
$$(I^{(t+1)}:_R I^{(1)})=I^{(t)}  \quad \text{for all } t\geq 1,$$   
 where $I^{(t)}=\bigcap_{\mathfrak{p}\in \mathrm{Min}(I)}(I^tR_\mathfrak{p}\cap R)$ denotes the  $t$-th symbolic  power  of $I$.

Now, assume that $I$ is a monomial ideal in the polynomial ring
$R=K[x_1,\ldots,x_n]$ over a field $K$, where $x_1,\ldots,x_n$ are
indeterminates. Even in this setting, the persistence property does not hold
universally. In fact, a counterexample exists even among square-free monomial
ideals, as shown in \cite{KSS}.  Generally, monomial ideals play a central role in understanding the deep connections
between combinatorics and commutative algebra. This interaction allows
algebraic techniques to be applied to combinatorial problems and, conversely,
combinatorial insights to inform algebraic questions. Particularly, one of the early contributors to this area is Villarreal, who introduced the
concept of edge ideals. Specifically, for a simple finite undirected graph $G$,
the \emph{edge ideal} $I(G)$, introduced in \cite{VI1}, is generated by the
monomials $x_i x_j$, where $\{i,j\}$ is an edge of $G$. 

 To date, numerous works employing combinatorial methods have been devoted to identifying classes of monomial ideals that satisfy the 
 persistence property (see \cite{HQ, MMV, N2, N3}), the strong persistence property (see \cite{ANR1, ANR2, BNT, NKA}), 
 and the symbolic strong persistence property (see \cite{KNT, NKRT}).

The main objective of this article is to introduce and develop a unified and
general theory of persistence for \emph{filtrations of ideals}. Given a
filtration $\mathcal{F}=\{I_i\}_{i\in\mathbb{N}}$ of a commutative ring $R$,
we define the notions of \emph{$\mathcal{F}$-persistence} and
\emph{$\mathcal{F}$-strong persistence} as natural extensions of the classical
persistence and strong persistence properties. Within this framework, the
classical notions of persistence, strong persistence, and symbolic strong
persistence are recovered as particular cases corresponding to the
filtrations given by ordinary powers and symbolic powers of an ideal. In particular, 
motivated by the theory of symbolic powers, we further introduce the
\emph{symbolic filtration} $\mathcal{F}_{\mathrm{sym}}$ associated with  a given
filtration $\mathcal{F}$. This paper has two main results. First, we show that if $\mathcal{F}$ is strongly
persistent, then $\mathcal{F}_{\mathrm{sym}}$ is strongly persistent (refer to 
Corollary~\ref{SPP-SSPP}). Second, we prove that if $\mathcal{F}$ is strongly
persistent, then $\mathcal{F}$ is persistent (cf.
Corollary~\ref{SPP-PP}). We conclude the paper by investigating how filtrations
behave under standard operations on rings, such as direct sums. In particular, as an application, we show that the direct sum filtration $\mathcal{F} \bigoplus\mathcal{G}$ is strongly persistent if and only if both $\mathcal{F}$ and $\mathcal{G}$ are strongly persistent, see Proposition \ref{DirectSum-SPP}.

 Throughout this text, the symbol $\mathbb{N}$ stands for the nonnegative integers. Also, all rings in this paper are commutative and Noetherian. 


\section{Filtrations of ideals}
In this section, we will introduce and study the concept of $\mathcal{F}$-Strong Persistence and $\mathcal{F}$-Persistence, where $\mathcal{F}$ 
 is a filtration. These concepts generalize the strong persistence property and the persistence property.
 To accomplish this, we begin with the following definition. 


\begin{definition}
Let $R$ be a ring. Then a  family of ideals $\mathcal{F}=\{I_n\}_{n\in\mathbb{N}}$ of $R$ is called a \emph{filtration} of $R$ if it satisfies the following conditions:
\begin{itemize}
\item[(i)]  $I_0=R$; 
    \item[(ii)] $I_{n+1}\subseteq I_n$ for all $n\in\mathbb{N}$;
    \item[(iii)]  $I_n I_m \subseteq I_{n+m}$ for all $n,m\in\mathbb{N}$.
\end{itemize}
\end{definition}


\begin{example} (\cite[Exercise 4.3.45]{V1})
\em{
Let $I$ be an ideal of a  ring $R$. Then  the following families define filtrations of $R$:
\begin{enumerate}[(a)]
\item $I_i = I^i$, the ordinary powers of $I$;
\item $I_i = \overline{I^i}$, the integral closure of $I^i$;
\item $I_i = I^{(i)}$, the symbolic powers of $I$;
\item $I_n = \bigoplus_{i \ge n} R_i$, where $R = \bigoplus_{i \ge 0} R_i$ is a graded ring.
\end{enumerate}
}
\end{example}


The following proposition plays a key role in the proof of Proposition \ref{Pro.Symbolic.1}.  

\begin{proposition} \label{Pro. Min}
Let $\mathcal{F}=\{I_i\}_{i\in\mathbb{N}}$ be a filtration of a  ring $R$.  
Then $\mathrm{Min}(I_i)=\mathrm{Min}(I_1)$ for all   $i\geq 1,$ where $\mathrm{Min}(L)$  denotes the set of  minimal prime ideals of an ideal $L$. 
\end{proposition}

\begin{proof}
Fix $i\geq 1$.   First, let $\mathfrak{p}\in \mathrm{Min}(I_1)$.   Due to $I_i \subseteq I_1$ and $I_1 \subseteq \mathfrak{p}$, we can deduce that 
$I_i \subseteq \mathfrak{p}$. In addition, it follows from the definition that  $I_1^i \subseteq I_i$.  
Since $I_i \subseteq \mathfrak{p}$, there exists  some  $\mathfrak{p}' \in  \mathrm{Min}(I_i)$ such that 
$I_i \subseteq \mathfrak{p}' \subseteq \mathfrak{p}$.  Because $I_1^i \subseteq I_i \subseteq \mathfrak{p}'$, we get  $I_1 \subseteq \mathfrak{p}' \subseteq \mathfrak{p}$. 
  By the minimality of $\mathfrak{p}$ over $I_1$, we can conclude that $\mathfrak{p}'=\mathfrak{p}$. This implies that $\mathfrak{p}\in \mathrm{Min}(I_i)$. 
  Hence,   $\mathrm{Min}(I_1) \subseteq   \mathrm{Min}(I_i)$.

Conversely, let $\mathfrak{p}\in \mathrm{Min}(I_i)$. Once again, since $I_1^i \subseteq I_i$, we obtain $I_1 \subseteq \mathfrak{p}$.  
 Hence there exists  some  $\mathfrak{p}'\in \mathrm{Min}(I_1)$ such that  $I_1 \subseteq \mathfrak{p}' \subseteq \mathfrak{p}$. 
 Due to $I_i \subseteq I_1$, we have  $I_i \subseteq \mathfrak{p}' \subseteq \mathfrak{p}$. This implies 
 that  $\mathfrak{p}'=\mathfrak{p}$  by the minimality of $\mathfrak{p}$ over $I_i$.  Therefore, $\mathfrak{p}\in \mathrm{Min}(I_1)$, and consequently 
$\mathrm{Min}(I_i) \subseteq   \mathrm{Min}(I_1)$. This finishes the proof. 
\end{proof}


\bigskip

Motivated by the notion of symbolic powers, we introduce the \emph{symbolic filtration}
$\mathcal{F}_{\mathrm{sym}}$ associated with   a given filtration $\mathcal{F}$ in Proposition \ref{Pro.Symbolic.1}. 
 To do this, we need to recall the next definition.

\begin{definition}
Let $R$ be a  ring  and let $I$ be an ideal of $R$.  
Assume that
\[
I=\bigcap_{j=1}^r Q_j
\]
is a minimal primary decomposition of $I$.  
We define the \emph{first symbolic power} of $I$ by
\[
I^{(1)} := \bigcap_{\sqrt{Q_j}\in \mathrm{Min}(I)} Q_j.
\]
\end{definition}

In particular, it is not hard to show that we always have  $$I^{(1)} =\bigcap_{\mathfrak p \in \mathrm{Min}(I)}
\left(I R_{\mathfrak p} \cap R\right).$$

 Before presenting the next proposition, let us recall the following theorem.
 
 \begin{theorem}(\cite[Theorem 3.61]{Sharp})  \label{PAT}
 Let $P_1, \ldots, P_n$, where $n > 2$, be ideals of the commutative ring $R$
such that at most two of $P_1, \ldots, P_n$ are not prime.
Let $S$ be an additive subgroup of $R$ which is closed under multiplication
(for example, $S$ could be an ideal of $R$ or a subring of $R$).
Suppose that
\[
S \subseteq \bigcup_{i=1}^n P_i.
\]
Then $S \subseteq P_j$ for some $j$ with $1 \le j \le n$.
 \end{theorem}


\begin{proposition} \label{Pro.Symbolic.1}
Let $\mathcal{F}=\{I_i\}_{i\in\mathbb{N}}$ be a filtration of $R$.  
Then the family
\[
\mathcal{F}_{\mathrm{sym}} := \{I_{(i)}\}_{i\in\mathbb{N}},
\quad \text{where } I_{(i)} := I_i^{(1)} \;  \; \text{and} \; \; I_{(0)}:=R, 
\]
is also a filtration of $R$.
\end{proposition}

\begin{proof}
We verify the defining properties of a filtration. By the definition of the
first symbolic power and   Proposition \ref{Pro. Min}, we get 
\[
I_{(i)}=I_i^{(1)}
= \bigcap_{\mathfrak p \in \mathrm{Min}(I_i)} (I_i R_{\mathfrak p} \cap R)
= \bigcap_{\mathfrak p \in \mathrm{Min}(I_1)} (I_i R_{\mathfrak p} \cap R), \; \text{where}\; i\geq 1.
\]
Since $I_{i+1} \subseteq I_i$, localization gives that 
\[
I_{i+1} R_{\mathfrak p} \subseteq I_i R_{\mathfrak p}
\quad \text{for all } \mathfrak p \in \mathrm{Min}(I_1).
\]
Taking intersections over all such $\mathfrak p$, we obtain $I_{(i+1)} \subseteq I_{(i)}.$  It  remains to prove that  
$I_{(i)} I_{(j)} \subseteq I_{(i+j)}$ for all $i,j \geq 1$.  Let
$
\mathrm{Min}(I_1)=\{\mathfrak{p}_1,\ldots, \mathfrak{p}_r\}.
$
Consider the following  minimal primary decompositions
\[
I_i=\bigcap_{k=1}^s Q_{ik}, \qquad
I_j=\bigcap_{\ell=1}^t Q_{j\ell}, \qquad
I_{i+j}=\bigcap_{m=1}^u Q_{(i+j)m}.
\]
By Proposition \ref{Pro. Min}, without loss of generality, we may assume that
\[
\mathfrak{p}_k=\sqrt{Q_{ik}}=\sqrt{Q_{jk}}=\sqrt{Q_{(i+j)k}}
\quad \text{for all } 1\leq k \leq r.
\]
This implies that 
\[
I_{(i)}=\bigcap_{k=1}^r Q_{ik} \quad \text{and} \quad I_{(j)}=\bigcap_{k=1}^r Q_{jk}.
\]
For each $k$ with $r<k\le s$, choose an element $x_k \in \sqrt{Q_{ik}} \setminus (\mathfrak{p}_1\cup\cdots\cup \mathfrak{p}_r)$ 
(Theorem \ref{PAT} guarantees the existence of such an element), 
and for each $\ell$ with $r<\ell\le t$, pick an element $y_\ell \in \sqrt{Q_{j\ell}} \setminus (\mathfrak{p}_1\cup\cdots\cup \mathfrak{p}_r)$ 
(Theorem \ref{PAT} guarantees the existence of such an element). 
On account of  $x_k\in \sqrt{Q_{ik}}$, there exists $a_k>0$ such that
$x_k^{a_k}\in Q_{ik}$ for all $r<k\le s$.
Similarly, since $y_\ell\in \sqrt{Q_{j\ell}}$, there exists $b_\ell>0$ such that
$y_\ell^{b_\ell}\in Q_{j\ell}$ for all $r<\ell\le t$. Take an element  $z\in I_{(i)}I_{(j)}$, and define
\[
z' := z\,
\bigl(x_{r+1}^{a_{r+1}}\cdots x_s^{a_s}\bigr)
\bigl(y_{r+1}^{b_{r+1}}\cdots y_t^{b_t}\bigr).
\]
Then $z'\in I_i I_j\subseteq I_{i+j}$. Here, fix $1\leq k\leq r$. Since $Q_{(i+j)k}$ is $\mathfrak{p}_k$-primary and
$x_{k'}^{a_{k'}}\notin \mathfrak{p}_k$ for all $k'\in\{r+1,\ldots,s\}$, it follows that
\[
z\,\bigl(y_{r+1}^{b_{r+1}}\cdots y_t^{b_t}\bigr)\in Q_{(i+j)k}.
\]
Again, since $y_{\ell'}^{b_{\ell'}}\notin \mathfrak{p}_k$ for all
$\ell'\in\{r+1,\ldots,t\}$ and $Q_{(i+j)k}$ is $\mathfrak{p}_k$-primary, we conclude that
$z\in Q_{(i+j)k}$. As this holds for every $1\leq k \leq r$, we obtain 
\[
z\in \bigcap_{k=1}^r Q_{(i+j)k} = I_{(i+j)}.
\]
This gives rise to $I_{(i)} I_{(j)} \subseteq I_{(i+j)},$ which completes the proof.
\end{proof}


We end this section with the following result.

\begin{proposition} \label{LIJ}
Let $\mathcal{F}=\{I_i\}_{i\in\mathbb{N}}$ be a filtration of a  ring $R$.  
Then
\[
I_{i-j} \subseteq (I_{i} : I_j)
\quad \text{for all } i,j \in\mathbb{N}, i-j>0.
\]
\end{proposition}

\begin{proof}
By definition of a filtration, for all  $i,j \in\mathbb{N}$ with  $i-j>0$, we have 
\[
I_j I_{i-j} \subseteq I_i.
\]
Hence, for every $x\in I_{i-j}$ and every $y\in I_j$, we obtain $xy\in I_{i}$, which implies  that  
$x\in (I_{i}:I_j)$.   Therefore, we can conclude that  $I_{i-j} \subseteq (I_{i}:I_j)$.
\end{proof}


\section{Strongly  Persistent filtrations}
 
Several of the arguments developed in this section are inspired by, and may be viewed as natural generalizations and adaptations of, the techniques introduced in the classical setting of ordinary powers of ideals, as studied in \cite{HQ}. In particular, our use of colon ideals, localization at prime ideals, and the characterization of associated primes follows the same philosophy underlying the persistence theory for powers, while requiring additional care to address the greater flexibility of arbitrary filtrations. The proofs presented here extend these ideas beyond the rigid structure of powers, preserving their core mechanisms while adapting them to the broader framework of filtrations. We begin with the following definitions.

 \begin{definition}
Let $\mathcal{F}=\{I_i\}_{i\in\mathbb{N}}$ be a filtration of a ring $R$.  
We say that $\mathcal{F}$ is \emph{strongly persistent} if
\[
(I_{i+1}:I_1)=I_{i}
\quad \text{for all } i\in\mathbb{N}.
\]
\end{definition}

\begin{definition}
Let $\mathcal{F}=\{I_i\}_{i\in\mathbb{N}}$ be a filtration of a  ring $R$.  
We say that $\mathcal{F}$ is \emph{persistent} if
\[
\mathrm{Ass}(R/I_i)\subseteq \mathrm{Ass}(R/I_{i+1})
\quad \text{for all } i\in\mathbb{N}.
\]
\end{definition}

These  notions generalize  the classical concept of persistence for powers and symbolic
powers of an ideal. Indeed, when $\mathcal{F}=\{I^i\}_{i\in\mathbb{N}}$ is the filtration
given by the ordinary powers of an ideal $I$, the condition
\[
(I^{i+1}:I)=I^i \text{ for } i\geq 1,
\]
recovers the usual definition of strongly persistence for powers. In the particular case of the filtrations given by the ordinary powers
$\{I^i\}_{i\ge 0}$ and the symbolic powers $\{I^{(i)}\}_{i\ge 0}$ of an ideal $I$. Similarly, when
$\mathcal{F}=\{I^{(i)}\}_{i\in\mathbb{N}}$ is the symbolic filtration associated with  $I$,
the equality
\[
(I^{(i+1)}:I)=I^{(i)} \text{ for } i\geq 1,
\]
corresponds to symbolic strong persistence. Hence, persistence of filtrations provides a
unified framework that encompasses both ordinary and symbolic powers as particular
cases.  In the case of ordinary powers and symbolic powers of an ideal, both the strong
persistence property and the symbolic strong persistence property imply the following 
\[
(I^i : I^j) = I^{\,i-j}
\quad \text{for all } i\ge j\ge 1,
\]
and 
\[
(I^{(i)} : I^{(j)}) = I^{(\,i-j)}
\quad \text{for all } i\ge j\ge 1,
\]
see \cite[Proposition 2.1]{RNA}  and \cite[Proposition 12]{RT} for more information. 

Following these results, we show in the subsequent proposition that this property also holds in the general setting of arbitrary filtrations.


\begin{proposition}
Let $\mathcal{F} = \{I_i\}_{i \ge 1}$ be a filtration of a ring $R$.  Then, for all $i \ge j \ge 1$,
 $(I_i : I_j) = I_{i-j}$ if and only if 
  $(I_{k+1} : I_1) = I_k$   for  all $k \geq 1.$ 
\end{proposition}

\begin{proof}
The forward implication is obvious. To establish the converse, we proceed by induction on $j \geq 1$.
 Let $j=1$. By our hypothesis, for all $i \ge 1$, we have $(I_i : I_1) = I_{i-1}.$ Hence, the statement holds for $j=1$.

Now, assume that for some $j \ge 1$, we have $(I_i : I_j) = I_{i-j} \quad \text{for all } i \ge j.$ 
We show that $(I_i : I_{j+1}) = I_{i-(j+1)}$ for all  $i \ge j+1.$ To do this, we first pick  $x \in I_{i-(j+1)}$. Since $\{I_k\}_{k \ge 1}$ is a 
filtration, we have $I_m I_n \subseteq I_{m+n}$ for all $m,n \ge 1$. Taking $m = i-(j+1)$ and $n = j+1$ yields
 $I_{i-(j+1)} I_{j+1} \subseteq I_i.$ This gives that  $x I_{j+1} \subseteq I_i$, and so $x \in (I_i : I_{j+1})$. Therefore, we obtain 
\[
I_{i-(j+1)} \subseteq (I_i : I_{j+1}).
\]
To complete the proof, let   $x \in (I_i : I_{j+1})$, i.e., $x I_{j+1} \subseteq I_i$.  It follows from the  hypothesis that $(I_{j+1} : I_1) = I_j$, and so 
$I_j I_1 \subseteq I_{j+1}$. Thus, for any $a \in I_j$, we have $x(a I_1) \subseteq x I_{j+1} \subseteq I_i$. This implies that 
$ (x a) I_1 \subseteq I_i.$   By the base case $(I_i : I_1) = I_{i-1}$, it follows that 
 $x a \in I_{i-1}$ for all $a \in I_j$. This yields that $x I_j \subseteq I_{i-1}.$  Therefore,  $x \in (I_{i-1} : I_j).$ 
 By the induction hypothesis on $j$ (note $i-1 \ge j$), we have  $(I_{i-1} : I_j) = I_{(i-1)-j} = I_{i-(j+1)}.$ 
 We thus get $x \in I_{i-(j+1)}$, and hence $(I_i : I_{j+1}) \subseteq I_{i-(j+1)},$ as required. 
  Combining the inclusions, we obtain  $(I_i : I_{j+1}) = I_{i-(j+1)}.$

This completes the induction step, and by induction on $j$, the   claim    holds for all $i \ge j \ge 1$.
\end{proof}


To establish Theorem \ref{IIsym}, we need the following auxiliary result. 

\begin{proposition}\label{Pro.Local}(\cite[Proposition 4.1]{NKRT})
Let $I$ be an ideal in a commutative Noetherian ring $R$. Also, let $I=Q_1 \cap \cdots \cap Q_t \cap Q_{t+1} \cap \cdots \cap Q_r$ be a minimal primary decomposition of $I$ with $\mathfrak{p}_i=\sqrt{Q_i}$ for  $i=1, \ldots, r$, and $\mathrm{Min}(I)=\{\mathfrak{p}_1, \ldots, \mathfrak{p}_t\}$. Then $I_{\mathfrak{p}_i}=({Q_i})_{\mathfrak{p}_i}$ for  $i=1, \ldots, t$.
\end{proposition}

\begin{theorem}\label{IIsym}
Let  $I,J,$ and $L$ be ideals of a ring $R$  such that
\[
\mathrm{Min}(I)=\mathrm{Min}(J)=\mathrm{Min}(L) \quad\text{and}\quad  (I:J)=L.
\]
Then
\[
(I^{(1)} : J^{(1)})= L^{(1)}.
\]
\end{theorem}
\begin{proof}
Set $\mathrm{Min}(I):=\{\mathfrak{p}_1,\ldots, \mathfrak{p}_r\}.$ Consider the following  minimal primary decompositions
\[
I=\bigcap_{k=1}^s Q_{1k},
\qquad
J=\bigcap_{\ell=1}^t Q_{2\ell},
\qquad
L=\bigcap_{m=1}^u Q_{3m}.
\]
Because  $\mathrm{Min}(I)=\mathrm{Min}(J)=\mathrm{Min}(L)$, without loss of generality, we may assume that
\[
\mathfrak{p}_k=\sqrt{Q_{1k}}=\sqrt{Q_{2k}}=\sqrt{Q_{3k}}
\quad \text{for all }  1\leq k \leq r.
\]
 We thus obtain the following 
\[
I^{(1)}=\bigcap_{k=1}^r Q_{1k}
\qquad\text{and}\qquad J^{(1)}=\bigcap_{k=1}^r Q_{2k}.
\]
For each $k$ with $r<k\le s$, select  an element
 $x_k \in \sqrt{Q_{1k}} \setminus (\mathfrak{p}_1\cup\cdots\cup \mathfrak{p}_r)$
  (the existence of such an element follows from Theorem \ref{PAT}). 
Since $x_k\in \sqrt{Q_{1k}}$, there exists $a_k>0$ such that
\[
x_k^{a_k}\in Q_{1k}
\quad \text{for all } r<k\le s.
\]
Now, let $x\in (I^{(1)}:J^{(1)})$.  Then $xy\in I^{(1)}$ for every $y\in J^{(1)}$. Since $J\subseteq J^{(1)}$,
it follows that $xy\in I^{(1)}$ for all  $y\in J.$ Therefore, we get 
\[
x\bigl(x_{r+1}^{a_{r+1}}\cdots x_s^{a_s}\bigr)y \in I
\quad \text{for all } y\in J,
\]
which shows that  $x\bigl(x_{r+1}^{a_{r+1}}\cdots x_s^{a_s}\bigr)\in (I:J).$ Due to  $(I:J)=L$, this implies that 
$
x\bigl(x_{r+1}^{a_{r+1}}\cdots x_s^{a_s}\bigr)\in L.
$
In particular, we can deduce that  
\[x\bigl(x_{r+1}^{a_{r+1}}\cdots x_s^{a_s}\bigr)\in Q_{3k}
\quad \text{for all } 1\leq k \leq r.
\]
Here, fix $1\leq k \leq r$. In light of  $Q_{3k}$ is $\mathfrak{p}_k$-primary and
$x_{k'}\notin \mathfrak{p}_k$ for all $k'\in\{r+1,\ldots,s\}$, one can conclude that $x\in Q_{3k}.$ 
By virtue of  this holds for every $1\leq k \leq r$, we conclude that $$x\in \bigcap_{k=1}^r Q_{3k} = L^{(1)}.$$
Consequently, we get   $(I^{(1)}:J^{(1)})\subseteq L^{(1)}.$ 

We now prove the reverse inclusion $L^{(1)} \subseteq (I^{(1)} : J^{(1)})$.
Let $z \in L^{(1)}$ and $y \in J^{(1)}$. Fix $1 \le k \le r$.
Localizing at $\mathfrak p_k$, one can derive from Proposition \ref{Pro.Local} that 
\[
I_{\mathfrak p_k} = (Q_{1k})_{\mathfrak p_k}, \quad
J_{\mathfrak p_k} =(Q_{2k})_{\mathfrak p_k}, \quad
L_{\mathfrak p_k} = (Q_{3k})_{\mathfrak p_k}.
\]
Since $(I:J)=L$, localization yields that $L_{\mathfrak p_k} = (I_{\mathfrak p_k} : J_{\mathfrak p_k}).$
Hence,
\[
(Q_{3k})_{\mathfrak p_k} \cdot (Q_{2k})_{\mathfrak p_k}
\subseteq (Q_{1k})_{\mathfrak p_k}.
\]
Contracting back to $R$ and using that $Q_{1k}$, $Q_{2k}$, and $Q_{3k}$ are  $\mathfrak p_k$-primary, we obtain
$
Q_{3k} Q_{2k} \subseteq Q_{1k}.
$
Therefore, since $z \in Q_{3k}$ and $y \in Q_{2k}$, we have $zy \in Q_{1k}$.
As this holds for every $1 \le k \le r$, it follows that
\[
zy \in \bigcap_{k=1}^r Q_{1k} = I^{(1)}.
\]
Thus $z \in (I^{(1)} : J^{(1)})$, proving the desired inclusion.
This finishes the proof. 
\end{proof}


The following result follows immediately from  Theorem \ref{IIsym}.

\begin{corollary} \label{SPP-SSPP}
Let $\mathcal{F}=\{I_i\}_{i\in\mathbb{N}}$ be a filtration of a  ring $R$, and let
 $\mathcal{F}_{\mathrm{sym}}=\{I_{(i)}\}_{i\in\mathbb{N}}$ 
be the symbolic filtration associated with  $\mathcal{F}$. Then  if $\mathcal{F}$ is strongly persistent, then $\mathcal{F}_{\mathrm{sym}}$
    is strongly persistent.
\end{corollary}


These results show that strong persistence is a robust property under the passage to
symbolic filtrations. In particular, the theory of persistent filtrations provides a
common framework extending the classical notions of persistence and symbolic
persistence for powers of ideals.

The properties of strong persistence and persistence property similarly to symbolic strong persistence are particular cases of the above concept for the filtration  $\{I^i\}_{i\in\mathbb{N}}$ and $\{I^{(i)}\}_{i\in\mathbb{N}}$,  respectively.

\bigskip

\begin{lemma} \label{ass(I:f)}
 Let $\mathcal{F}=\{I_k\}_{k\in\mathbb{N}}$ be a filtration of a  ring $R$.
Fix $k\in\mathbb{N}$ and let $\mathfrak{p}\in V(I_1)$, where $V(I_1)$ denotes  the variety of  $I_1$.  Then an element
 $f\in ((I_{k})_\mathfrak{p} : \mathfrak m_{\mathfrak{p}}) \setminus (I_k)_\mathfrak{p}$ if and only if 
 $\bigl((I_k)_\mathfrak{p} : f\bigr)=\mathfrak{m}_\mathfrak{p},$ 
where $(I_k)_\mathfrak{p}$ denotes the localization of $I_k$ at $\mathfrak{p}$, and $\mathfrak{m}_{\mathfrak{p}}$
is the maximal ideal of the local ring $R_\mathfrak{p}$.
\end{lemma}
\begin{proof} 
 $(\Rightarrow)$   Suppose that $f\in \bigl((I_k)_\mathfrak{p} : \mathfrak{m}_\mathfrak{p}\bigr)\setminus (I_k)_\mathfrak{p}.$  
 Then, by definition, we have  $\mathfrak{m}_\mathfrak{p} f \subseteq (I_k)_\mathfrak{p}$, which implies that 
$\mathfrak m_{\mathfrak{p}} \subseteq \bigl((I_k)_{\mathfrak{p}} : f\bigr).$

Conversely, let $x\in \bigl((I_k)_{\mathfrak{p}} : f\bigr)$, so that $xf\in (I_k)_{\mathfrak{p}}$.
If $x\notin \mathfrak m_{\mathfrak{p}}$, then $x$ is a unit in the local ring $R_{\mathfrak{p}}$.
Hence, $f=x^{-1}(xf)\in (I_k)_{\mathfrak{p}},$ which contradicts the assumption that $f\notin (I_k)_{\mathfrak{p}}$.
Therefore, we must have  $x\in  \mathfrak m_{\mathfrak{p}}$, and so we obtain $\bigl((I_k)_{\mathfrak{p}} : f\bigr)\subseteq \mathfrak m_{\mathfrak{p}}.$ 
 Combining both inclusions, we conclude that
\[
\bigl((I_k)_{\mathfrak{p}} : f\bigr)=\mathfrak m_{\mathfrak{p}}.
\]
$(\Leftarrow)$ Conversely, assume that $\bigl((I_k)_{\mathfrak{p}} : f\bigr)=\mathfrak m_{\mathfrak{p}}.$ 
This  yields that $\mathfrak m_{\mathfrak{p}} f \subseteq (I_k)_\mathfrak{p}$, and hence 
 $f\in \bigl((I_k)_{\mathfrak{p}} : \mathfrak m_{\mathfrak{p}}\bigr)$. Moreover, if $f\in (I_k)_\mathfrak{p}$, then
 $\bigl((I_k)_{\mathfrak{p}} : f\bigr)=R_{\mathfrak{p}},$ which contradicts the assumption that
$\bigl((I_k)_{\mathfrak{p}} : f\bigr)=\mathfrak m_{\mathfrak{p}}$.
Therefore, we must have  $f\in \bigl((I_k)_{\mathfrak{p}} : \mathfrak m_{\mathfrak{p}}\bigr)\setminus (I_k)_{\mathfrak{p}}.$ 
This completes our argument. 
\end{proof}

\begin{remark}
The previous characterization shows that there exists an element
 $f\in \bigl((I_k)_{\mathfrak{p}} : \mathfrak m_{\mathfrak{p}}\bigr)\setminus (I_k)_{\mathfrak{p}}$
if and only if  $\mathfrak m_{\mathfrak{p}}=\bigl((I_k)_{\mathfrak{p}} : f\bigr).$
Equivalently, this condition holds if and only if $\mathfrak m_\mathfrak{p}$ is an associated
prime of the $R_{\mathfrak{p}}$-module $R_{\mathfrak{p}}/(I_k)_{\mathfrak{p}}$. In particular, there exists such an element
$f$ if and only if $\mathfrak{p} \in \mathrm{Ass}(R/I_k)$.
\end{remark}


\begin{proposition}\label{spp-persis}
 Let $\mathcal{F}=\{I_k\}_{k\in\mathbb{N}}$ be a filtration of a  ring $R$. Fix  $k\in\mathbb{N}$. 
  Assume that for any $\mathfrak{p}\in \mathrm{Ass}(R/I_k)$ and any element
 $f\in \bigl((I_k)_{\mathfrak{p}} : \mathfrak m_{\mathfrak{p}}\bigr)\setminus (I_k)_{\mathfrak{p}},$
there exists an element $g\in (I_1)_{\mathfrak{p}}$ such that
 $fg\notin (I_{k+1})_{\mathfrak{p}}.$  Then    $\mathrm{Ass}(R/I_k)\subseteq \mathrm{Ass}(R/I_{k+1}).$
\end{proposition}

\begin{proof}
Take  $\mathfrak{p}\in \mathrm{Ass}(R/I_k)$. Then, by localization, the maximal ideal
$\mathfrak m_{\mathfrak{p}}$ belongs to $\mathrm{Ass}(R_{\mathfrak{p}}/(I_k)_{\mathfrak{p}})$. Hence,  there exists an element
  $f\in R_{\mathfrak{p}}$   such that  $\bigl((I_k)_{\mathfrak{p}} : f\bigr)=\mathfrak m_{\mathfrak{p}}.$ 
It follows from  Lemma \ref{ass(I:f)} that $f\in \bigl((I_k)_{\mathfrak{p}} : \mathfrak m_{\mathfrak{p}}\bigr)\setminus (I_k)_{\mathfrak{p}}.$
In addition, by our hypothesis, there exists an element $g\in (I_1)_{\mathfrak{p}}$ such that
  $fg\notin (I_{k+1})_{\mathfrak{p}}.$ Since $fg\notin (I_{k+1})_{\mathfrak{p}}$, we get  
  $\bigl((I_{k+1})_{\mathfrak{p}} : fg\bigr)\subseteq \mathfrak m_{\mathfrak{p}}.$ We claim 
  $\mathfrak m_{\mathfrak{p}}\subseteq \bigl((I_{k+1})_{\mathfrak{p}} : fg\bigr).$ By virtue of  $\mathcal F$ is a filtration, we have
  $I_k\cdot I_1 \subseteq I_{k+1}$. Localizing at $\mathfrak p$, this gives $(I_k)_{\mathfrak p}\cdot (I_1)_{\mathfrak p}
\subseteq (I_{k+1})_{\mathfrak p}.$
Let $x\in \mathfrak m_{\mathfrak p} = \bigl((I_k)_{\mathfrak p}:f\bigr)$.
Then $xf\in (I_k)_{\mathfrak p}$, and since $g\in (I_1)_{\mathfrak p}$, we can derive that 
\[
xfg \in (I_k)_{\mathfrak p}(I_1)_{\mathfrak p}
\subseteq (I_{k+1})_{\mathfrak p}.
\]
Hence, $x\in \bigl((I_{k+1})_{\mathfrak p}:fg\bigr)$, and therefore $\mathfrak m_{\mathfrak p}\subseteq \bigl((I_{k+1})_{\mathfrak p}:fg\bigr).$
Consequently, we obtain  $\bigl((I_{k+1})_{\mathfrak{p}} : fg\bigr)=\mathfrak{m}_{\mathfrak{p}},$
which implies that $\mathfrak{m}_{\mathfrak{p}}\in \mathrm{Ass}(R_{\mathfrak{p}}/(I_{k+1})_{\mathfrak{p}})$.
Finally, by the correspondence between associated primes under localization,
we conclude that  $\mathfrak{p}\in \mathrm{Ass}(R/I_{k+1}),$  and hence
$\mathrm{Ass}(R/I_k)\subseteq \mathrm{Ass}(R/I_{k+1}),$  as desired.
 \end{proof}
 
 
 \begin{theorem}\label{Ratliff-colon}
Let $\mathcal{F}=\{I_k\}_{k\in\mathbb{N}}$ be a filtration of a ring $R$.
Fix $k\ge 1$. Then $(I_{k+1} : I_1)= I_k$ if and only if for every $\mathfrak{p}\in \mathrm{Ass}(R/I_k)$ and every element
  $f\in \bigl((I_k)_{\mathfrak{p}} : \mathfrak m_{\mathfrak{p}}\bigr)\setminus (I_k)_{\mathfrak{p}},$ 
there exists an element $g\in (I_1)_{\mathfrak{p}}$ such that $fg\notin (I_{k+1})_{\mathfrak{p}}.$
\end{theorem}

\begin{proof}
We first assume that $(I_{k+1}:I_1)= I_k$ for all $k\ge 1.$ Fix $k\geq 1$. Localizing at a prime ideal $\mathfrak{p}$, we can deduce that 
  $((I_{k+1})_{\mathfrak{p}} : (I_1)_{\mathfrak{p}})= (I_k)_{\mathfrak{p}}$. 
Let $\mathfrak{p}\in \mathrm{Ass}(R/I_k)$. Then the maximal ideal $\mathfrak m_{\mathfrak{p}}$ belongs to
$\mathrm{Ass}(R_{\mathfrak{p}}/(I_k)_{\mathfrak{p}})$, and hence there exists an element
 $f\in \bigl((I_k)_{\mathfrak{p}} : \mathfrak m_{\mathfrak{p}}\bigr)\setminus (I_k)_{\mathfrak{p}}.$
 
Suppose, by contradiction, that $fg\in (I_{k+1})_{\mathfrak{p}} \quad \text{for all } g\in (I_1)_{\mathfrak{p}}.$ 
Then $f\in ((I_{k+1})_{\mathfrak{p}} : (I_1)_{\mathfrak{p}}) = (I_k)_{\mathfrak{p}}$, contradicting the choice of $f$.
Therefore, there exists $g\in (I_1)_{\mathfrak{p}}$ such that $fg\notin (I_{k+1})_{\mathfrak{p}}$.

Conversely, assume that the stated condition holds, but there exists some $k \geq 1$ such that $(I_{k+1} : I_1)\neq  I_k.$  
 Then the $R$-module $M:=(I_{k+1} : I_1)/I_k$ is nonzero. Hence, $\mathrm{Ass}(M)\neq \emptyset$. Due to 
$\mathrm{Ass}(M) \subseteq \mathrm{Supp}(M)$, this implies that $\mathrm{Supp}(M)\neq \emptyset$.  Choose a minimal prime ideal ${\mathfrak{p}}$ in the support of this module. Localizing at ${\mathfrak{p}}$, we obtain that the following $R_{\mathfrak{p}}$-module
\[
M_{\mathfrak{p}}=\left(\frac{(I_{k+1} : I_1)}{I_k}\right)_{\mathfrak{p}}=\frac{\left({(I_{k+1})_{\mathfrak{p}} : (I_1)_{\mathfrak{p}}}\right)}
{({I_k})_{\mathfrak{p}}}
\]
is nonzero.  
Since $\mathfrak p$ is a minimal prime in $\mathrm{Supp}(M)$, localization at $\mathfrak p$ eliminates all other primes in the support. Indeed, if $\mathfrak q\in \mathrm{Supp}_{R_{\mathfrak p}}(M_{\mathfrak p})$, then $\mathfrak q=\mathfrak q'R_{\mathfrak p}$ for some prime ideal $\mathfrak q'\subseteq \mathfrak p$ such that $\mathfrak q'\in \mathrm{Supp}(M)$. By the minimality of $\mathfrak p$ in $\mathrm{Supp}(M)$, we must have $\mathfrak q'=\mathfrak p$, and hence $\mathfrak q=\mathfrak m_{\mathfrak p}$. Therefore,
$
\mathrm{Supp}_{R_{\mathfrak p}}(M_{\mathfrak p})=\{\mathfrak m_{\mathfrak p}\}.
$
As $M_{\mathfrak p}\neq 0$, it follows that $\mathfrak m_{\mathfrak p}\in \mathrm{Ass}_{R_{\mathfrak p}}(M_{\mathfrak p}).$ 
Hence, there exists a nonzero element $\overline{f}\in M_{\mathfrak p}$ such that
 $\mathfrak m_{\mathfrak p}\,\overline{f}=0.$ Let $f\in \bigl((I_{k+1})_{\mathfrak p}:(I_1)_{\mathfrak p}\bigr)$ be a representative of $\overline{f}$.
Then $\overline{f}\neq 0$ implies $f\notin (I_k)_{\mathfrak p}$, and the annihilation condition yields
$
\mathfrak m_{\mathfrak p} f \subseteq (I_k)_{\mathfrak p}.
$
In particular, we have  $f\in \bigl((I_k)_{\mathfrak{p}} : \mathfrak{m}_{\mathfrak{p}}\bigr)\setminus (I_k)_{\mathfrak{p}}.$ 
By our hypothesis, there exists an element  $g\in (I_1)_{\mathfrak{p}}$ such that
 $fg\notin (I_{k+1})_{\mathfrak{p}},$  which contradicts the fact that $f\in ((I_{k+1})_{\mathfrak{p}} : (I_1)_{\mathfrak{p}})$.
Accordingly, we deduce that $(I_{k+1}:I_1)= I_k$, as required. 
\end{proof}


\begin{corollary} \label{SPP-PP}
Let $\mathcal{F}=\{I_i\}_{i\in\mathbb{N}}$ be a filtration of a  ring $R$.  
If $\mathcal{F}$ is strongly persistent, then $\mathcal{F}$ is persistent.
\end{corollary}

\begin{proof}
The claim is an immediate consequence of Proposition \ref{spp-persis} and Theorem \ref{Ratliff-colon}.
    \end{proof}


\section{Filtrations under operations}

In this  short  section, we investigate how filtrations behave under standard operations on rings, such as direct sums. We also show that filtrations on a product of rings can be decomposed into filtrations on each factor. In particular, as an application, we show that the direct sum filtration $\mathcal{F} \bigoplus\mathcal{G}$ is strongly persistent if and only if both $\mathcal{F}$ and $\mathcal{G}$ are strongly persistent.

\begin{proposition}\label{direct-sum-filtration}
Let $\mathcal F=\{I_i\}_{i\in \mathbb{N}}$ and $\mathcal G=\{J_i\}_{i\in \mathbb{N}}$
be filtrations of the rings $R$ and $R'$, respectively.
Then the family $\{I_i\oplus J_i\}_{i\in \mathbb{N}}$ is a filtration of the ring
$R\oplus R'$.
\end{proposition}

\begin{proof}
Recall that the ring structure on $R\oplus R'$ is given componentwise by
\[
(r_1,r_1')(r_2,r_2')=(r_1r_2,r_1'r_2')
\quad \text{for all } r_1,r_2\in R,\; r_1',r_2'\in R'.
\]
Since $\mathcal F$ and $\mathcal G$ are filtrations, we have
 $I_0=R$ , $J_0=R',$ and
\[
I_i I_j \subseteq I_{i+j}, \qquad
J_i J_j \subseteq J_{i+j}
\quad \text{for all } i,j\in \mathbb{N}.
\]
Clearly, $I_0\oplus J_0 = R\oplus R'.$
Let $(a,b)\in I_i\oplus J_i$ and $(c,d)\in I_j\oplus J_j$.
Then $a\in I_i$, $c\in I_j$, $b\in J_i$, and $d\in J_j$, and
 $(a,b)(c,d)=(ac,bd),$  where $ac\in I_i I_j$ and $bd\in J_i J_j$.
This shows that
\[
(I_i\oplus J_i)(I_j\oplus J_j)\subseteq I_i I_j\oplus J_i J_j.
\]
Conversely, every element of $I_i I_j\oplus J_i J_j$ is a finite sum of elements
of the form $(ac,bd)$ with $a\in I_i$, $c\in I_j$, $b\in J_i$, and $d\in J_j$,
and hence belongs to $(I_i\oplus J_i)(I_j\oplus J_j)$.
Therefore,  we can conclude that 
\[
(I_i\oplus J_i)(I_j\oplus J_j)= I_i I_j\oplus J_i J_j.
\]
It follows that
\[
(I_i\oplus J_i)(I_j\oplus J_j)
\subseteq I_{i+j}\oplus J_{i+j}
\quad \text{for all } i,j\in \mathbb{N},
\]
and hence $\{I_i\oplus J_i\}_{i\in \mathbb{N}}$ is a filtration of $R\oplus R'$, as claimed. 
\end{proof}

 
  Before formulating the next proposition, we require to recall the following facts:
  
  \begin{fact} \label{fact.1}
Let $R$ and $R'$ be rings. For ideals $I, I' \subseteq R$ and $J, J' \subseteq R'$, we have
\[
(I \oplus J)(I' \oplus J') = I I' \oplus J J'.
\]
\end{fact}

\begin{fact}\label{fact.2}
Let $R$ and $R'$ be rings. For ideals $A, A' \subseteq R$ and $B, B' \subseteq R'$, we have
\[
A \oplus B \subseteq A' \oplus B' \quad \Longleftrightarrow \quad A \subseteq A' \text{ and } B \subseteq B'.
\]
\end{fact}

\begin{proposition}\label{filtration-splits}
Let $\mathcal F=\{\mathcal I_i\}_{i\in \mathbb{N}}$ be a filtration of the ring
$R\oplus R'$.
Then for each $i\in \mathbb{N}$ there exist ideals $I_i\subseteq R$ and
$J_i\subseteq R'$ such that
\[
\mathcal I_i= I_i\oplus J_i.
\]
Moreover, the families $\{I_i\}_{i\in \mathbb{N}}$ and
$\{J_i\}_{i\in \mathbb{N}}$ are filtrations of $R$ and $R'$, respectively.
\end{proposition}

\begin{proof}
Since every ideal of the direct product $R\oplus R'$ is of the form $I\oplus J$,
where $I\subseteq R$ and $J\subseteq R'$ are ideals, for each $i\in\mathbb{N}$
there exist ideals $I_i\subseteq R$ and $J_i\subseteq R'$ such that $\mathcal I_i= I_i\oplus J_i.$
As $\mathcal F$ is a filtration, we have $\mathcal I_0=R\oplus R'$. Hence,
$I_0\oplus J_0 = R\oplus R'.$
For any $a\in R$, the element $(a,0)$ belongs to $R\oplus R'$, and thus
$(a,0)\in I_0\oplus J_0$, which implies $a\in I_0$. Therefore $I_0=R$.
Similarly, considering elements of the form $(0,b)$ shows that $J_0=R'$.
 Furthermore, since $\mathcal F$ is a filtration, we have
$\mathcal I_{i+1}\subseteq \mathcal I_i$ for all $i$, and so   
\[
I_{i+1}\oplus J_{i+1}\subseteq I_i\oplus J_i.
\]
Consequently, by virtue of Fact  \ref{fact.2}, we obtain   $I_{i+1}\subseteq I_i$ and $J_{i+1}\subseteq J_i$.
Finally, the multiplicative property
$\mathcal I_i\,\mathcal I_j\subseteq \mathcal I_{i+j}$ gives that 
\begin{equation}
\label{eq:1}
(I_i\oplus J_i)(I_j\oplus J_j)\subseteq I_{i+j}\oplus J_{i+j}.
\end{equation}
In addition, according to Fact \ref{fact.1}, we have 
\begin{equation}
\label{eq:2}
(I_i \oplus J_i)(I_j \oplus J_j) = I_i I_j \oplus J_i J_j.
\end{equation}
It   follows now from  (\ref{eq:1}), (\ref{eq:2}),   and  Fact   \ref{fact.2}  that 
\[
I_iI_j\subseteq I_{i+j}
\quad\text{and}\quad
J_iJ_j\subseteq J_{i+j}.
\]
Therefore, $\{I_i\}_{i\in\mathbb{N}}$ is a filtration of $R$ and
$\{J_i\}_{i\in\mathbb{N}}$ is a filtration of $R'$.
\end{proof}

We conclude the discussion presented in this section with the following proposition.

\begin{proposition} \label{DirectSum-SPP}
    Let $\mathcal{F}=\{I_i\}_{i\in\mathbb{N}}$ and $\mathcal{G}=\{J_i\}_{i\in\mathbb{N}}$ be filtrations of ideals of a commutative ring $R$. 
Then the direct sum filtration $\mathcal{F} \bigoplus\mathcal{G}$ is strongly persistent if and only if both $\mathcal{F}$ and $\mathcal{G}$ are strongly persistent.
\end{proposition}

\begin{proof}
According to Proposition \ref{direct-sum-filtration}, the family $\mathcal{F} \bigoplus\mathcal{G}=\{I_i\oplus J_i\}_{i\in \mathbb{N}}$ is a filtration of  
$R\oplus R$. We first assume that  both $\mathcal{F}$ and $\mathcal{G}$ are strongly persistent. 
 Let $(x,y)\in (I_{i+1}\oplus J_{i+1} : I_1\oplus J_1)$.   Then
\[
(x,y)(a,b)\in I_{i+1}\oplus J_{i+1}
\quad \text{for all } (a,b)\in I_1\oplus J_1.
\]
In particular, for every $a\in I_1$ and $b\in J_1$ we have $xa\in I_{i+1}$ and  $yb\in J_{i+1}.$ Hence,
$x\in (I_{i+1}:I_1)$ and $y\in (J_{i+1}:J_1).$ 
Since the filtrations $\mathcal{F}$ and $\mathcal{G}$ are strongly persistent, it follows that
 $x\in I_i$ and  $y\in J_i$. This implies that $(x,y)\in I_i\oplus J_i.$ Therefore,  the direct sum filtration $\mathcal{F} \bigoplus\mathcal{G}$ is strongly
  persistent.

Conversely, assume that $x\in (I_{i+1}:I_1)$ and $y\in (J_{i+1}:J_1).$ 
Then for every $(a,b)\in I_1\oplus J_1$ we have
\[
(x,y)(a,b)=(xa,yb)\in I_{i+1}\oplus J_{i+1},
\]
which implies that  $(x,y)\in (I_{i+1}\oplus J_{i+1}: I_1\oplus J_1).$  On account of  $\mathcal{F} \bigoplus\mathcal{G}$ is strongly
  persistent, we have $(I_{i+1}\oplus J_{i+1}: I_1\oplus J_1)= I_i\oplus J_i.$  This gives rise to 
  $(x,y)\in I_i\oplus J_i,$ and hence $x\in I_i$ and $y\in J_i$. This means that both $\mathcal{F}$ and $\mathcal{G}$ are strongly persistent, as claimed. 
\end{proof}


\bigskip

\textbf{Acknowledgments}

The authors  would like to express sincere gratitude to Professor Rafael H. Villarreal for his valuable suggestion of extending persistence properties through filtrations, which inspired the development of the present work.

\bigskip
\noindent\textbf{ORCID}
\begin{itemize}
\item  Mehrdad Nasernejad:~~ 0000-0003-1073-1934
\item Jonathan Toledo:~~0000-0003-3274-1367
\end{itemize}


\end{document}